\documentclass[10pt]{article}
\usepackage{amssymb}
\usepackage[dvipsnames]{xcolor}

\usepackage{amsmath, amssymb, amsthm, mathtools}
\usepackage{tikz}
\usetikzlibrary{arrows.meta}
\usetikzlibrary{positioning}
\usepackage{caption, subcaption, float}
\usepackage[makeroom]{cancel}
\usepackage[all, knot]{xy}
\usepackage{graphicx}
\usepackage{xcolor}
\usepackage{todonotes}
\usepackage{enumitem}
\DeclareSymbolFont{symbolsC}{U}{txsyc}{m}{n}
\DeclareMathSymbol{\strictif}{\mathrel}{symbolsC}{74}

\newtheorem{theorem}{Theorem}[section]

\newtheorem*{theorem*}{Theorem} 
\newtheorem*{lemma*}{Lemma}

\theoremstyle{definition}
\newtheorem{definition}[theorem]{Definition}

\newtheorem{proposition}[theorem]{Proposition}

\newtheorem{corollary}[theorem]{Corollary}

\theoremstyle{remark}

\theoremstyle{remark}

\newcommand\restr[2]{{
\left.\kern-\nulldelimiterspace 
#1 
\vphantom{\big|} 
\right|_{#2} 
}}

\newcommand{\A}{\mathcal{A}}

\newcommand{\K}{\mathcal{K}}

\newcommand{\M}{\mathcal{M}}
\newcommand{\N}{\mathbb{N}}

\newcommand{\U}{\mathbb{U}}

\let\phi\varphi
\title{On the sunflower property and the galah property}
\author{Cheng Liao}
\date{}
\begin{document}

\maketitle

\begin{abstract}
Sunflowerability, or the infinite sunflower property, was introduced and studied by Ackerman, Karker and Mirabi as a structural generalization of the well-known $\Delta$-system lemma for sets. It turns out that for relational Fra\"iss\'e limits with strong amalgamation, this property is equivalent to the so-called galah property, which was introduced by Sullivan and Winkel as an asymmetric variation of indivisibility. This paper is about these two properties and is divided into three parts. In the first part, we show that the conjecture proposed by Ackerman, Karker and Mirabi about the infinite sunflower property in higher dimensions is far from being true by proving that no infinite structure has the infinite $n$-sunflower property in dimension $k$ for any $n,k\geq 2$. In the second part, we give a complete characterization of the galah property for Henson directed graphs, homogeneous metric spaces and homogeneous ultrametric spaces, thereby answering the second question asked by Sullivan and Winkel. The third part contains several additional results about the finite sunflower property, including a strengthening of recent results about indivisibility for some classes of undirected graphs obtained by Guingona \textit{et al.}.
\end{abstract}

\section{Introduction}

A \textit{sunflower} or a \textit{$\Delta$-system} is a collection of sets in which any two distinct sets have the same intersection (called the root). The well-known $\Delta$-system lemma proved by Erd\"os and Rado states that for every infinite set of sets of size $n$, there is a subset of the same cardinality that forms a $\Delta$-system. Instead of considering pure sets, Ackerman, Karker and Mirabi thought about the possibility of imposing some extra structure on the set we consider \cite{1}. This led them to the following definition:

\begin{definition}\cite[Def. 6]{1}
    For any $n\in\mathbb{N}^+$, an infinite first-order structure $\mathcal{M}$ is \textit{$n$-sunflowerable} or has the \textit{infinite $n$-sunflower property} if for any structure on $n$-sets $\M'$ with $\M'\cong \M$, $\M'$ contains a substructure $\mathcal{N}$ such that $\mathcal{N}\cong \M$ and $\mathcal{N}$ is a sunflower (as a set of $n$-sets). $\M$ has the \textit{infinite sunflower property} if $\M$ has \textit{infinite $n$-sunflower property} for any $n\in \mathbb{N}^+$.
\end{definition}

It is immediate that every infinite structure $\M$ has the infinite $1$-sunflower property. And it is proved that for relational Fra\"iss\'e limits with strong amalgamation, the infinite 2-sunflower property is equivalent to the infinite sunflower property \cite{2}. In order to prove this, the authors introduced what they called ``the galah property" which is an asymmetric variation of the well-studied partition property \textemdash  indivisibility.

\begin{definition}\label{indiv}
    A structure $\M$ is \textit{indivisible} if for any partition $(A,B)$ of $\M$, either $A$ contains a copy of $\M$ or $B$ contains a copy of $\M$.

    A structure $\M$ has the \textit{galah property} if for any partition $(A,B)$ of $\M$, either $A\cong \M$ or $B$ contains a copy of $\M$ \cite[Def. 2.1]{2}.
\end{definition}

Indivisibility was introduced by Fra\"iss\'e in the fifties and has been well studied since then because of its connection to Ramsey theory, topological dynamics and so on \cite{s,ms,cs,g}. It is clear that the galah property implies indivisibility, and it is also easy to see that in general indivisibility does not imply the galah property (e.g., a countably infinitely set with an infinite equivalence relation such that each equivalence class is infinite). What makes the galah property interesting is the following result:

\begin{theorem}\cite[Prop 2.15 and Thm A]{2}
    Let $\M$ be a relational Fra\"iss\'e limit with strong amalgamation. Then the following are equivalent:
\begin{enumerate}
    \item [1)] $\M$ has the galah property.
    \item [2)] $\M$ has the infinite 2-sunflower property.
    \item [3)] $\M$ has the infinite sunflower property.
    \item [4)] $\M$ is locally replicable\footnote{See Definition 3.3.}.
    \item [5)] $\M$ has the canonical infinite point-Ramsey property (i.e., every coloring of $\chi: M \rightarrow \mathbb{N}$ has a monochromatic or heterochromatic\footnote{Every two different elements have different colors.} copy of $\M$).
\end{enumerate}
    
\end{theorem}

In this paper, we answer two questions about the infinite sunflower property and the galah property raised in \cite{1} and \cite{2}. In particular, we disprove the conjecture that the infinite sunflower property is equivalent to its higher dimensional versions \cite[Conj. 62]{1} by showing that no infinite structure can have the infinite 2-sunflower property in dimension higher than 2. We also give a complete characterization of the galah property for $H_\mathcal{T}$ \textemdash the Fra\"iss\'e limit of all finite $\mathcal{T}$-free directed graphs, homogeneous metric spaces and homogeneous ultrametric spaces. This answers the second question in \cite{2} and helps us understand the place of those structures with the galah property in the landscape of indivisible structures. In the end, we will also prove some results about the finite sunflower property \textemdash the finite analogue of the infinite sunflower property. In particular, we introduce the \textit{large-girth amalgamation property}\footnote{See Definition 4.9.} for classes of finite structures, and use it to prove the finite sunflower property for the class of all finite $\A$-free graphs where $\A$ is a set of finite \textit{non-2-covered}\footnote{See remarks before Theorem \ref{non-2}.} graphs, which strengthens results about indivisibility obtained in \cite{g}.

\section{Higher-dimensional infinite sunflower property}

We start with the definition of the infinite sunflower property in higher-dimensional settings. Throughout this paper, substructures are always induced substructures. Hence in particular, subgraphs are  induced subgraphs. In the following, $ \mathcal{P}_n(S)$ denotes the set of all $n$-subsets of $S$, and $|S|$ denotes the cardinality of $S$. For  any structure $\M$, $dom(\M)$ is its underlying set which we may also denote as $M$ when no confusion is caused.

\begin{definition}
   Let $\M$ be an infinite structure and $k,n\in\mathbb{N}^+$. We say that $\M$ has the infinite $n$-sunflower property in dimension $k$ if whenever $f: M^k\rightarrow \mathcal{P}_n(S)$ is an injection where $S$ is a set, there is a substructure $\M_0 \subseteq \M$ such that $\M_0\cong \M$ and $f[M^k_0]$ is a sunflower.
\end{definition}

In \cite{1}, it was conjectured that for any infinite structure $\M$ and any $k,n\in\mathbb{N}^+$, $\M$ has the infinite $n$-sunflower property in dimension $k$ if and only if $\M$ has the infinite $n$-sunflower property. In this section, we show that this conjecture is false.

First, the following proposition is more or less obvious.

\begin{proposition}
 For any structure $\M$ and any $n,k\in\mathbb{N}^+$, if $\M$ has the infinite $n+1$-sunflower property in dimension $k$, then $\M$ has the infinite $n$-sunflower property in dimension $k$.
\end{proposition}

\begin{proof}
    Let $n,k\in\mathbb{N}^+$ be arbitrary,  and suppose $\M$ has the infinite $n+1$-sunflower property in dimension $k$. Let $f: M^k\rightarrow \mathcal{P}_n(S)$ be an injection where $S$ is a set, then choose an element $t$ not in $S$ and define $g: M^k\rightarrow \mathcal{P}_{n+1}(S\cup \{t\})$ as follows: $g(a_1,...,a_k)=f(a_1,...,a_k)\cup\{t\}$. Clearly $g$ is an injection, and thus by assumption, there exists $\M_0\subseteq \M$ such that $\M_0\cong \M$ and $g[M_0^k]$ is a sunflower, say with $X$ as the common intersection. Then $f[M_0^k]$ is a sunflower with $X\setminus \{t\}$ as the common intersection. Hence $\M$ has the infinite $n$-sunflower property in dimension $k$.
\end{proof}

Now we prove that for any $k\geq 2$, no infinite structure has the infinite $2$-sunflower property in dimension $k$.

\begin{proposition}
For any $k\geq 2$, no infinite structure has the infinite $2$-sunflower property in dimension $k$.   
\end{proposition}

\begin{proof}
Let $k\geq 2$ be arbitrary, we divide the proof into several steps for convenience. 

\textbf{Claim 1:} If $\M$ has the infinite $2$-sunflower property in dimension $k$, then it has the following property: for any $\chi:M^k \rightarrow \mathbb{N}$ as a coloring, there is $\M_0\subseteq \M$ such that $\M_0\cong\M$ and $\chi|_{M_0^k}$ is heterochromatic or monochromatic ($\star$).

Suppose $\M$ has the infinite $2$-sunflower property in dimension $k$, let $\chi: M^k\rightarrow \mathbb{N}$ be a coloring. Without loss of generality we assume that $\mathbb{N}\cap M=\emptyset$, and we define $f: M^k\rightarrow \mathcal{P}_2(S)$ as follows: $f(a_1,...,a_k)=\{(a_1,...,a_k), \chi(a_1,...,a_n)\}$. Clearly $f$ is an injection, and thus by assumption, there exists $\M_0\subseteq \M$ such that $\M_0\cong \M$ and $f|_{M_0^k}$ is a sunflower. If the common intersection is the empty set, then $\chi|_{M_0^k}$ is heterochromatic; if the common intersection is nonempty, then it must consist of a single color $i\in\mathbb{N}$, and thus $\chi|_{M_0^k}$ is monochromatic. Hence $\M$ satisfies $(\star)$.

\textbf{Claim 2:} If $\M$ is a structure satisfying $(\star)$, then all $k$-tuples of $\M$ have the same quantifier-free type.

Suppose $\M$ satisfies $(\star)$ while there exist $(a_1,...,a_k), (b_1,...,b_k)\in M^k$ such that $\mathrm{qftp}\footnote{``qftp" is short for the quantifier-free type. It is the set of all quantifier-free formulas that are true of the tuple.}(a_1,...,a_k)=p_1\not=p_2=\mathrm{qftp}(b_1,...,b_k)$, we define $\chi:M^k \rightarrow\mathbb{N}$ as follows: 
$\chi(\Bar{m})=\begin{cases}0 & \text{ if } \mathrm{qftp}(\Bar{m})=p_1 \cr 1 & \text{ if } \mathrm{qftp}(\Bar{m})=p_2 \cr 2 & \text{ otherwise }  \end{cases}$

Then by assumption, there exists $\M_0\subseteq \M$ such that $\M_0\cong \M$ and $\chi|{M_0^k}$ is heterochromatic or monochromatic. As $\M$ is infinite, it can't be heterochromatic and thus $\chi|{M_0^k}$ is monochromatic, say with color 0. Then for any $\Bar{m}\in M_0^k$, $\mathrm{qftp}(\Bar{m})=p_1$. However, $\M_0\cong \M$ and $(b_1,...,b_k)\in M^k$ with $\mathrm{qftp}(b_1,...,b_k)=p_2\not=p_1$, we get a contradiction. Similarly for being monochromatic with color 1 or 2. This proves the claim.

Now combining the above two claims together, we get that if $\M$ has the infinite $2$-sunflower property in dimension $k$, then all $k$-tuples of $\M$ have the same quantifier-free type. But any structure $\M$ whose size is bigger than 1, say $a\not=b\in M$ we have $(a,...,a)$ and $(a,b,...,b)$ as two $k$-tuples having  different quantifier-free types (note $k\geq 2$). Hence, no infinite structure has the infinite $2$-sunflower property in dimension $k$.
  
\end{proof}

Now combining the above two propositions together, we get the following immediately.

\begin{proposition}
   No infinite structure has the infinite $n$-sunflower property in dimension $k$ for $n,k\geq 2$.
\end{proposition}

Since there are many examples of infinite structures with the infinite $n$-sunflower property (e.g., infinite set with no extra structure), the conjecture mentioned above is false.

One may wonder if the problem is simply due to the fact that we are thinking about $k$-tuples instead of $k$-sets when moving towards higher-dimensional settings. Thus one may propose the following as an alternative higher-dimensional infinite sunflower property. We add $^*$ to distinguish it from the one in Def 2.1. 

\begin{definition}
    Let $\M$ be an infinite structure $n,k\in\mathbb{N}^+$ and $(M)^k$ denote the set of $k$-subsets of $M$. We say that $\M$ has the infinite $n$-sunflower$^*$ property in dimension $k$ if whenever $f: (M)^k\rightarrow \mathcal{P}_n(S)$ is an injection where $S$ is a set, there is a substructure $\M_0 \subseteq \M$ such that $\M_0\cong \M$ and $f[(M_0)^k]$ is a sunflower.
\end{definition}

Unfortunately, the alternative definition does not work.

\begin{proposition}
    No infinite structure has the infinite $n$-sunflower$^*$ property in dimension $k$ for any $n,k\geq 2$.
\end{proposition}

\begin{proof}
    First, the same proof as that in Proposition 2.2 gives us that if $\M$ has the infinite $n+1$-sunflower$^*$ property in dimension $k$, then $\M$ has the infinite $n$-sunflower property$^*$ in dimension $k$. Thus we can focus on the infinite $2$-sunflower property in dimension $k$ for $k\geq 2$. 

 Then we have the following claim, the proof of which is the same as that of the above Claim 1. 

 \textbf{Claim:} If $\M$ has the infinite $2$-sunflower$^*$ property in dimension $k$, then it has the following property: for any $\chi:(M)^k \rightarrow |M|$ as a coloring, there is $\M_0\subseteq \M$ such that $\M_0\cong\M$ and $\chi|_{(M_0)^k}$ is heterochromatic or monochromatic ($\triangle$).

Now let $n,g\geq 2$ be arbitrary, suppose an infinite structure $\M$ has the infinite $n$-sunflower$^*$ property in dimension $k$, then $\M$ has the infinite $2$-sunflower$^*$ property in dimension $k$. Therefore, by the claim, $\M$ satisfies $\triangle$. Now we define $\chi: (M)^k\rightarrow |M|$ as follows: let $(a_i)_{i<\kappa}$ be an enumeration of $M$ without repetition (thus $|M|=\kappa$), for any $\{a_{i_1},...,a_{i_k}\}\in (M)^k$, $\chi(\{a_{i_1},...,a_{i_k}\})= i_j$ where $i_j$ is the smallest index among $\{i_1,...,i_k\}$. It is easy to see that $\chi$ has no monochromatic or heterochromatic copy of $\M$ as $\M$ is infinite. Thus we get a contradiction, and this proves that there is no infinite structure with the infinite $n$-sunflower$^*$ property in dimension $k$ for any $n,k\geq 2$.
    
\end{proof}

\section{Homogeneous structures with the galah property}

As the galah property implies indivisibility, the authors in \cite{2} asked which indivisible Fra\"iss\'e structures considered in the literature possess the galah property. In this section, we answer this for those mentioned in the second question of \cite{2}. We recall very basic facts about \textit{Fra\"iss\'e classes}, which one can find in basic books in model theory (e.g., \cite{k}) and almost every paper about indivisible structures (e.g., \cite{s,ms,cs}). In the rest of this paper, all classes of structures are assumed to be closed under isomorphisms.

\begin{definition}
    A class of structures $\K$ is a \textit{Fra\"iss\'e class} if it has the following properties:

    \begin{enumerate}
        \item[1)] $\K$ is closed under taking substructures (\textit{Hereditary Property}).
        \item[2)] For any $A,B\in\K$, there is $C\in K$ such that both $A$ and $B$ embed into $C$  (\textit{Joint Embedding Property}).
        \item [3)] For any embeddings $f:A\rightarrow B$ and $g:A\rightarrow C$ where $A,B,C\in \K$, there is $D\in \K$ with embeddings $f':B\rightarrow D$ and $g':C\rightarrow D$ such that $f'\circ f=g'\circ g$ (\textit{Amalgamation Property}).
    \end{enumerate}    
\end{definition}

It is easy to see that if there is an initial element in $\K$ (e.g., directed and undirected simple graphs), then the third property implies the second one. For the amalgamation property, if one can have $D\in \K$ with embeddings $f':B\rightarrow D$ and $g':C\rightarrow D$ satisfy the extra condition that $ f'(B)\cap g'(C) =f'\circ f(A) (=g'\circ g(A))$, we say that $\K$ has \textit{strong amalgamation}\footnote{Some people call it \textit{disjoint amalgamation}.}. For a relational language $L$, we say $\K$ has \textit{free amalgamation} if one can have $D$ with embeddings $f':B\rightarrow D$ and $g':C\rightarrow D$ satisfy the extra condition for strong amalgamation and $R^D= R^{f'(B)}\cup R^{g'(C)}$ for any $R\in L$ where we view $f'(B)$ and $g'(C)$ as substructures of $D$.

For the following properties, people also call them ``$n$-ultrahomogeneous" and ``ultrahomogeneous" in the literature.

\begin{definition}
    A structure $\mathcal{M}$ is \textit{n-homogeneous} if every isomorphism between $n$-generated substructures of $\mathcal{M}$ extends to an automorphism of $\mathcal{M}$. $\mathcal{M}$ is \textit{homogeneous} if it is $n$-homogeneous for each $n\in \mathbb{N}$.
\end{definition}

For any structure $\M$, let $Age(\M)$ (called the \textit{age} of $\M$) denote the class of all finite structures that are embeddable into $\M$. The well-known Fr\"aiss\'e theorem states that each Fra\"iss\'e class $\K$ has a unique generic structure whose age is $\K$, and its uniqueness is characterized by its age and homogeneity:

\begin{theorem*}[Fr\"aiss\'e Theorem] Let $L$ be a countable relational\footnote{Actually ``relational" can be replaced by something more general, like locally finiteness. In this paper, we do not need the theorem in its most general form.} language and $\K$ be a Fra\"iss\'e class of finite $L$-structures, then there is a countable homogeneous structure $\M$ (called the \textit{Fra\"iss\'e limit}) such that $Age(\M)=\K$. And it is the unique countable homogeneous structure whose age is $\K$.      
\end{theorem*}

In fact, the Fra\"iss\'e limit $\M$ is also characterized by its age together with the extension property with respect to its age: for any embedding $f:A\rightarrow \M$ and $A\subseteq B$ where $A,B\in Age(\M)$, $f$ extends to an embedding of $B$ into $\M$. When the language is relational, one can even assume that $B=A\cup \{b\}$ where $b\not\in A$.

Now we introduce some definitions related to indivisibility and the galah property.

\begin{definition}
    For any structure $\M$, for any finite $A\subseteq \M$ and $a\in M\setminus A$, let $Age (a/A)$ denote the set of all finite structures embeddable into $\{m\in M\mid \mathrm{qftp}(m/A) =\mathrm{qftp}(a/A)\}$ (viewed as a substructure of $\M$). The \textit{rank} of $\M$ is $\{Age(a/A)\mid A\subseteq M \text{ is finite and } a\in M\setminus A\}$, and we say $\M$ is \textit{locally replicable} if $\{m\in M\mid \mathrm{qftp}(m/A) =\mathrm{qftp}(a/A)\}$ contains a copy of $\M$ for any finite $A\subseteq \M$ and $a\in M\setminus A$.
    
\end{definition}

The definition of \textit{rank}\footnote{Sauer called it ``\textit{the rank of types}". We drop ``types" here as it may cause confusion. For model-theorists, what Sauer called ``types" are actually special quantifier-free types.} was introduced by Sauer to give a characterization of indivisibility for countable oligomorphic homogeneous relational structures with free amalgamation: for those structures, indivisibility is equivalent to linearity of rank (under $\subseteq$) \cite[Thm 2.1]{s}. \textit{Locally replicability} was first introduced in \cite[Def. 13]{1} with the name ``\textit{universal duplication of quantifier-free types}". Here we follow the terminology in \cite[Def. 2.9]{2}.

\begin{proposition}\label{rank}
    A countable homogeneous structure $\M$ in binary relational language with strong amalgamation has the galah property if and only if the rank of $\M$ is singleton.
\end{proposition}

\begin{proof}
The left-to-right direction is obvious by Theorem 1.3. For the right-to-left direction, as the rank is singleton, for any $a\not=b\in M$, $\mathrm{qftp}(a/\emptyset)=\mathrm{qftp}(b/\emptyset)$. Thus for any $a\in M$, $Age(a/\emptyset)= Age(\M)$. The rank of $\M$ is just $\{Age(\M)\}$. Now the proof is the same as that of \cite[Lem 2.28]{2}. We give the details here for completeness.

For any finite $A\subseteq \M$ and $a\in M\setminus A$, let $\Gamma(a/A)=\{m\in M\mid \mathrm{qftp}(m/A) =\mathrm{qftp}(a/A)\}$. For any $A'\in Age(\M)$ such that $f: A'\rightarrow \Gamma(a/A) $ is an embedding and let $B=A'\cup \{b\}\in Age(\M)$. As $Age(a/A)=Age(\M)$, there is an embedding $g: B\rightarrow \Gamma(a/A)$. As the language is binary relational, there is an isomorphism from $f(A')\cup A$ onto $g(A')\cup A$ which maps $f(c)$ to $g(c)$ for any $c\in A'$ and fixes $A$ pointwise. By homogeneity of $\M$, the isomorphism extends to an automorphism of $\M$, denoted by $h$. Clearly $h(\Gamma(a/A))=\Gamma(a/A)$ as $h$ fixes $A$ pointwise. In particular, $h^{-1}(g(b))\in \Gamma(a/A)$. Then the map $f\cup\{(b,h^{-1}(g(b)))\}$ is an embedding of $B$ into $\Gamma(a/A)$ extending $f$. Therefore, $\Gamma(a/A)\cong \M$. This proves that $\M$ is locally replicable, and by Theorem 1.3, $\M$ has the galah property.

\end{proof}

By \cite[Thm 2.1]{s}, we know that for any countable homogeneous structure in finite binary relational language with free amalgamation, if its rank is linear but not singleton, then it is an indivisible structure without the galah property. For example, \cite[Exam 11.6]{s}.

Besides, the corresponding statement for $\omega$-categorical relational structures is also true. Using Theorem 1.3, the following is just \cite[Lem 2.27]{2} translated to Sauer's terminology.

\begin{proposition}
    Let $\M$ be a $\omega$-categorical\footnote{$\omega$-categoricity means that there is a unique countably infinite structure satisfying its first-order theory up-to-isomorphism. By the well-known Ryll-Nardzewski theorem, it is equivalent to the automorphism group of the structure being oligomorphic. For example, see \cite[Thm 4.3.1]{k}.} countable homogeneous structure in relational language with strong amalgamation. $\M$ has the galah property if and only if the rank of $\M$ is singleton.
\end{proposition}

By Proposition 3.5, we know that every $k$-uniform hypergraphs in \cite[Thm 2.5]{s} is an indivisible structure without the galah property as its rank is of size $n-k+1$ where $n>k\geq 2$. Besides, since the galah property implies indivisibility, Proposition 3.5 is a weaker version of \cite[Thm 2.6]{s} as we still need strong amalgamation.

\subsection{Homogeneous directed graphs}

The following definition of \textit{decompositions} was first introduced in \cite{ms} for tournaments.

\begin{definition}
    Let $G$ be a directed (simple) graph. A \textit{decomposition} of $G$ is a partition of the vertex set into four sets $A_1, B, A_2, C$ with $B\not=\emptyset$ such that there are directed edges from $x$ to $y$ and from $y$ to $z$ for any $x\in A_1$, $y\in B$ and $z\in A_2$, and there is no directed edge with one element in $B$ and one element in $C$. We write $G=(A_1,B,A_2,C)$ to indicate that partition.
\end{definition}

Note that when $G$ is a tournament, $C$ must be empty in a decomposition of $G$. We call a decomposition \textit{nontrivial} if $2\leq |B|$ and $1\leq |A_1\cup A_2\cup C|$. Now let $\mathcal{T}$ be a set of finite directed graphs such that no element of $\mathcal{T}$ embeds into another element of $\mathcal{T}$. Suppose that the class of all finite $\mathcal{T}$-free directed graphs has strong amalgamation with $H_{\mathcal{T}}$ as its Fra\"iss\'e limit (e.g., $\mathcal{T}$ is a set of finite tournaments), then we have the following characterization of the galah property for $H_{\mathcal{T}}$.

\begin{theorem}
$H_\mathcal{T}$ has the galah property if and only if every element of $\mathcal{T}$ with cardinality $\geq 3$ has no nontrivial decomposition.
\end{theorem}

\begin{proof}
    First, if $\mathcal{T}$ has an element with cardinality $2$, then that element must be a two-element directed graph with a directed edge. In that case $H_\mathcal{T}$ is a countably infinite independent set which clearly has the galah property. In the following, we assume every element of $\mathcal{T}$ has size $\geq 3$.

For the left-to-right direction, suppose  $H_\mathcal{T}$ has the galah property while there is $A\in \mathcal{T}$ such that $A$ has a nontrivial decomposition $(A_1,B,A_2,C)$ where $|B|\geq 2$ and $|A_1\cup A_2\cup C|\geq 1$. Let $b$ be an element of $B$. As $A\in \mathcal{T}$ is minimal by our assumption about $\mathcal{T}$, $\{b\}\cup A_1\cup A_2\cup C\in Age(H_\mathcal{T})$.  Without loss of generality we can assume $\{b\}\cup A_1\cup A_2\cup C\subseteq H_\mathcal{T}$ as a substructure. Now we consider the set $\Gamma(b/A_1\cup A_2\cup C)=\{m\in H_\mathcal{T}\mid \mathrm{qftp}(m/A_1\cup A_2\cup C)= \mathrm{qftp}(b/A_1\cup A_2\cup C)\}$. As $H_\mathcal{T}$ has the galah property, by Theorem 1.3, $H_\mathcal{T}$ is locally replicable. Thus $\Gamma(b/A_1\cup A_2\cup C)$ contains a copy of $H_\mathcal{T}$. As $|A_1\cup A_2\cup C|\geq 1$,  $B$ is a proper substructure of $A$ and thus $B\in Age(H_\mathcal{T})$. Therefore, $\Gamma(b/A_1\cup A_2\cup C)$ contains a copy of $B$. But then that copy together with $A_1\cup A_2\cup C$ gives us a copy of $A$ inside $H_\mathcal{T}$, contradicting the assumption that $A\in \mathcal{T}$ and $H_\mathcal{T}$ is $\mathcal{T}$-free. This proves the left-to-right direction.

For the right-to-left direction, suppose every element of $\mathcal{T}$ with cardinality $\geq 3$ has no nontrivial decomposition. Let $A\subseteq H_\mathcal{T}$ be finite and $a\in H_\mathcal{T}\setminus A$ be arbitrary. Suppose $Age(a/A)\subsetneq Age(H_\mathcal{T})$, by transitivity\footnote{As we are considering directed simple graphs, every element has the same quantifier-free type over the empty set.}, we know that $A\not=\emptyset$. There exists $B\in Age(H_\mathcal{T})$ such that $B$ is not embeddable into $\Gamma(a/A)$. Define $A_1$, $A_2$ and $C$ as follows: 

\noindent $A_1=\{x\in A\mid \text{there is a directed edge from }x \text{ to } a\}$\\ $A_2=\{y\in A\mid \text{there is a directed edge from }a \text{ to } y\}$\\ $C=\{w\in A\mid \text{there is no directed edge connecting }a \text{ and } w\}$.

Clearly $A=A_1\cup A_2\cup C$. Without loss of generality we may assume $B\cap A=\emptyset$. Now as $B$ is not embeddable into $\Gamma(a/A)$, the directed graph $A\cup B$ with edges $E(A)\cup\{(x,y)\mid x\in A_1, y\in B \}\cup \{(y,z)\mid y\in B, z\in A_2\}$ is not $\mathcal{T}$-free. Hence there is $T\in \mathcal{T}$ such that $A\cup B$ contains a copy of $T$, and we can assume $T\subseteq A\cup B$ without loss of generality. As $A\cup\{a\}\subseteq H_\mathcal{T}$, it does not contain a copy of $T$. Now clearly $|T\cap B|\geq 2$, for otherwise, $T\subseteq A\subseteq H_\mathcal{T}$ if $T\cap B=\emptyset$ or $T\cong A\cup\{a\}$ if $|T\cap B|=1$, and we get a contradiction in either case. Now if $T\cap A=\emptyset$, then $T\subseteq B\in Age(H_\mathcal{T})$, a contradiction as $H_\mathcal{T}$ is $\mathcal{T}$-free and $T\in\mathcal{T}$. Therefore, $|T\cap A|=|T\cap (A_1\cup A_2\cup C))|\geq 1$, and thus $(A_1\cap T, B\cap T, A_2\cap T, C\cap T)$ is a nontrivial decomposition of $T$. This contradicts the assumption about $\mathcal{T}$. Therefore, $Age(a/A)=Age(H_\mathcal{T})$. By Proposition \ref{rank}, $H_\mathcal{T}$ has the galah property.
    
\end{proof}

The simplest example of a tournament which has no nontrivial decomposition is the directed cycle on three vertices. A more complicated example is the tournament on five vertices $\{a,b,c,d,e\}$ with edge relation $\{(a,b),(b,c),(c,d),(d,e)\\,(e,a),(c,a),(a,d),(b,e),(d,b),(e,c)\}$\footnote{Intuitively, it is a directed cycle with a directed pentagram inside.}. Now let $\mathcal{T}$ be a set of finite tournaments such that no one embeds into another one. It is well-known that the class of all finite $\mathcal{T}$-free directed graphs forms a free amalgamation class (see, for example, \cite[Lem 3.2.7]{th}). Thus we have the following.

\begin{corollary}\label{nontrivial}
    Let $\mathcal{T}$ be a set of finite tournaments, the Fra\"iss\'e limit of the class of all finite $\mathcal{T}$-free directed graphs $H_\mathcal{T}$ has the galah property if and only if every element of $\mathcal{T}$ with cardinality $\geq 3$ has no nontrivial decomposition.
\end{corollary}

In \cite{ms}, the authors introduced the so-called \textit{derived sets} and define a preorder on them. They then proved that indivisibility of $H_\mathcal{T}$ where $\mathcal{T}$ is a finite set of finite tournaments is equivalent to the preorder being total. Using their terminology, Corollary \ref{nontrivial} can be easily translated to the statement that the galah property of $H_\mathcal{T}$ is equivalent to the cardinality of the set of derived sets being 1\footnote{If one allows the empty set as a tournament, then the cardinality should be two with $\mathcal{T}$ and the set whose element is the directed graph with only one vertex. Besides, one may note that for our result, we do not need to assume that $\mathcal{T}$ is finite.}. Hence, for any  finite set of finite tournaments $\mathcal{T}$ with no element embeds into another one, if the set of derived sets is a total preorder with size $\geq 2$, then $H_\mathcal{T}$ is an indivisible structure without the galah property.

\subsection{Homogeneous metric spaces}

In this section, we give a complete answer to the question about the galah property of homogeneous metric spaces.

Let us recall some basic facts about metric spaces. A metric space $\M$ is a pair $(M,d)$ where $M$ is a set and $d$ is a metric on $M$ with values in $\mathbb{R}_{\geq 0}$. For any $x\in M$, we define $Spec(\M, x)=\{d(x,a)\mid a\in M \}$ and the \textit{spectrum} is the set $\bigcup\{Spec(\M,x)\mid x\in M\}$. As one would expect, when talking about indivisibility or the galah property for metric spaces, we are considering isometric copies and isometries.

Let $0\in V\subseteq \mathbb{R}_{\geq 0}$ be a countable set with at least two elements. If it satisfies the four-values condition\footnote{This is a technical condition which turns out to be equivalent to the class of all finite metric spaces whose spectrum is contained in $V$ being a Fra\"iss\'e class. For its precise definition, see \cite{cs} for example.}, we let $\mathbb{U}_V$ be the Fra\"iss\'e limit of the class of all finite metric spaces whose spectrum is contained in $V$. 

\begin{proposition} \label{prop: galah if s to 2s}
 $\U_V$ has the galah property if $V\subseteq \{0\}\cup [s,2s]$\footnote{$V$ always satisfies the four-values condition in this case as one can see for example that the Fra\"iss\'e limit always exists.} for some $s>0$.   
\end{proposition}

\begin{proof}
    Note that $\U_V$ is characterized by the following property: for any $n\in\mathbb{N}^+$, any $s_1,...,s_n\in V\setminus \{0\}$ and any distinct $n$ elements $v_1,...,v_n\in \U_V$, there exists $w\in \U_V$ such that $d(w,v_i)=s_i$ for any $1\leq i\leq n$ ($\star$). Namely, any countable metric space whose spectrum is $V$ with the property $(\star)$ is isometric to $\U_V$ (just by Fra\"iss\'e's theory and the fact that triangle inequality is always satisfied). In particular, for a nonempty subspace $A\subseteq \U_V$, if it satisfies $(\star)$, then its spectrum is clearly $V$ and thus isometric to $\U_V$.

    Now let $(A,B)$ be an arbitrary partition of $\U_V$, suppose both $A$ and $B$ do not satisfy $(\star)$. Thus there exist $s_1,...,s_n\in V\setminus\{0\}$ and
    $n$ distinct elements $a_1,...,a_n\in A$ such that there is no element $a$ of $A$ with $d(a,a_i)=s_i$ for any $1 \leq i\leq n$, and there exist $t_1,...,t_m\in V\setminus\{0\}$ and
    $m$ distinct elements $b_1,...,b_m\in B$ such that there is no element $b$ of $B$ with $d(b,b_j)=t_j$ for any $1 \leq j\leq m$ where $n,m\in\mathbb{N}^+$. Since $\U_V$ satisfies $(\star)$, there exists $v\in \U_V$ with $d(v,a_i)=s_i$ and $d(v,b_j)=t_j$ for any $1\leq i\leq n$ and $1\leq j\leq m$. As $(A,B)$ is a partition of $\U_V$, either $v\in A$ or $v\in B$. In either case, we get a contradiction. Therefore, if $A\not\cong \U_V$ (i.e., not isometric to $\U_V$), then $A$ does not satisfies $(\star)$. By what we've proved, we know that $B$ must satisfy $(\star)$. Thus $B$ is isometric to $\U_V$ and in particular contains an isometric copy of $\U_V$. This proves that $\U_V$ has the galah property.
\end{proof}

\begin{proposition}\label{s to 2s}
    If there exist $s,t\in V$ such that $s\not=0$ and $2s<t$, then $\U_V$ does not have the galah property.
\end{proposition}

\begin{proof}
    Pick an arbitrary element $a\in \U_V$, by transitivity, we know $Spec(\U_V,a)=V$. Now let $A=\{b\in U_V\mid d(a,b)=s\}$, and consider the partition $(U_V\setminus A, A)$ of $\U_V$. Clearly $U_V\setminus A$ is not isometric to $\U_V$ as $a\in U_V\setminus A$ and for any $x\in U_V\setminus A$, $d(a,x)\not=s$. For $A$, by triangle inequality, we have that for any $x,y\in A$, $d(x,y)\leq 2s$. Since $2s<t$, we know that $A$ does not contain an isometric copy of $\U_V$. Therefore, $\U_V$ does not have the galah property.
\end{proof}

\begin{theorem}
    $\U_V$ has the galah property if and only if $V\subseteq \{0\}\cup [s,2s]$ for some $s>0$.
\end{theorem}

\begin{proof}
    The right-to-left direction is given by Proposition \ref{prop: galah if s to 2s}. For the left-to-right direction, if $\U_V$ has the galah property,  pick an element $s\in V\setminus\{0\}$, by Proposition \ref{s to 2s}, we get that $V\setminus\{0\}\subseteq [s,2s]$. Thus $V\subseteq \{0\}\cup [s,2s]$.
\end{proof}

As far as the author knows, we do not have such a complete picture about indivisibility of $\U_V$'s. The best result known so far was proved in \cite{v} which says that if $V$ is finite, then $\U_V$ is indivisible.

\subsection{Homogeneous ultrametric spaces}

A metric space $\M=(M,d)$ is ultrametric if it satisfies the strong triangle inequality: for any $x,y,z\in M$, $d(x,z)\leq 
max\{d(x, y), d(y,z)\}$. Let $0\in V\subseteq \mathbb{R}_{\geq 0}$ be countable, the class of all finite ultrametric spaces whose spectrum is contained $V$ forms a Fra\"iss\'e class and thus has its Fra\"iss\'e limit $\U^u_V$, a countable homogeneous ultrametric space with $V$ as its spectrum.

Unlike homogeneous metric spaces, we do have a clear description of homogeneous ultrametric spaces and indivisible homogeneous ultrametric spaces. Let $\lambda$ be a countable linear order, define $\mathbb{N}^{[\lambda]}:=\{(a_\mu)_{\mu\in \lambda}\mid a_\mu\in \N \textit{ and } (a_\mu)_{\mu\in \lambda} \textit{ has finite}\\
\textit{support}\}$. Now let $f:\lambda\rightarrow \mathbb{R}_{>0}$ be a strictly decreasing map, we define $d_f(\Bar{a},\Bar{b})=f(\mu)$ where $\mu$ is the smallest index where $a_\mu\not=b_{\mu}$ for $\Bar{a}\not=\Bar{b}\in \N^{[\lambda]}$, and $d_f(\Bar{a},\Bar{a})=0$. Let $V$ be the image of $d_f$.

\begin{theorem}\cite[Prop 2.8 and Thm 2.13]{cs}\label{ultrahom}
 The space $(\mathbb{N}^{[\lambda]}, d_f)$ is the countable homogeneous ultrametric space with spectrum $V$.  If $|V|\geq 2$, then $(\mathbb{N}^{[\lambda]}, d_f)$ is indivisible if and only if $V$ is dually well-ordered.  
\end{theorem}

With this in mind, we can now give a complete characterization of the galah property for homogeneous ultrametric spaces as follows:

\begin{proposition}
    A countably infinite homogeneous ultrametric space $\U^u_V$ has the galah property if and only if $|V|=2$.
\end{proposition}

\begin{proof}
    For the left-to-right direction, suppose $\U^u_V$ has the galah property. As the galah property implies indivisibility, by Theorem \ref{ultrahom}, we know that $\U^u_V = (\mathbb{N}^{[\lambda]},d_f)$ where $\lambda$ is a countable well-order. Without loss of generality, we can assume $\lambda$ is a countable ordinal. Now suppose $|V|\geq 3$, then $\lambda\geq 2$. Let $s\in \mathbb{R}_{>0}$ be the largest element of $V$ (it exists as $\lambda$ is a well-order), and let $\Bar{b}=(0,1,0,...)\in \mathbb{N}^{[\lambda]}$. Now consider the set $X=\{(a_i)_{i\in\lambda}\in\mathbb{N}^{[\lambda]}\mid a_0\not=0\}\cup \{b\}$ and the partition $(X, U^u_V\setminus X)$ of $\U^u_V$. Clearly $X$ is not isometric to  $\U^u_V$ as $X\vDash \exists x \forall y(x\not= y\rightarrow d_f(x,y)=s)$ while $\U^u_V\not\vDash \exists x \forall y(x\not= y\rightarrow d_f(x,y)=s)$ (as for every element of $\U^u_V$, there is another element which differs from it at the second coordinate but not at the first coordinate). However, as $U^u_V\setminus X\vDash \forall x\forall y(d_f(x,y)<s)$ while $\U^u_V\not\vDash\forall x\forall y(d_f(x,y)<s)$, we also know that $U^u_V\setminus X$ does not contain an isometric copy of $\U^u_V$. Therefore, $\U^u_V$ does not have the galah property, contradicting our assumption. Hence $|V|<3$, and as $\U^u_V$ is infinite, $|V|>1$ and thus $|V|=2$.

For the right-to-left direction, suppose $|V|=2$, this means that for any $x\not=y\in \U^u_V$, we have $d_f(x,y)=s$ where $V=\{0,s\}$ and $s\not=0$. Then one can easily check that $\U^u_V$ has the galah property just by using the pigeonhole principle.
    
\end{proof}

\section{Finite sunflower property}

In the final section, we prove some results about finite sunflower property \textemdash the finite version of the infinite sunflower property. We start with its definition.

\begin{definition}\cite[Def. 49]{1}
   Let $n\in\mathbb{N}^+$, a class of finite structures $\mathcal{K}$ (in the same language) is said to have the finite $n$-sunflower property if for each $A\in \mathcal{K}$, there is $B\in\mathcal{K}$ such that for any structure on $n$-sets $B'$ with $B'\cong B$, we have that $B'$ has a substructure isomorphic to $A$ which is also a sunflower. $\mathcal{K}$ has the finite sunflower property if it has the finite $n$-sunflower property for any $n\in\mathbb{N}^+$.
\end{definition}

Using K\"onig's Lemma or compactness, one can easily show that for any  $n\in\mathbb{N}^+$, if a countable structure $\M$ has the infinite $n$-sunflower property, then $Age(\M)$ has the finite $n$-sunflower property \cite[Thm. 53]{1}.  However, the converse is not true. In particular, the class of all finite $K_n$-free graphs has the finite sunflower property while its Fra\"iss\'e limit does not have the infinite sunflower property \cite[Exam 3.18]{2}. 

It is not quite surprising that the finite sunflower property implies a stronger version of the Ramsey property.  

\begin{definition}
    Let $\K$ be a class of finite structures, $\K$ has the \textit{canonical point-Ramsey  property} if for any $A\in \K$, there is $B\in \K$ such that for every coloring $\chi: B\rightarrow \mathbb{N}$, there is a monochromatic or heterochromatic copy of $A$ inside $B$.
\end{definition}

\begin{proposition}\cite[Thm B]{2}
Let $\K$ be a class of finite structures, then if $\K$ has the finite $2$-sunflower property, then it has the canonical point-Ramsey property.   
\end{proposition}

Note that although in the paper \cite{2}, the above result is only stated for Fra\"iss\'e classes, it actually holds for any class of structures. It is also conjectured that for a Fra\"iss\'e class, the canonical point-Ramsey property is equivalent to the finite sunflower property \cite[Conj 4.1]{2}. 

\begin{proposition}\label{definable equi}
    Let $\M$ be a locally finite\footnote{Namely, finitely generated substructures are finite. In particular, $\M$ is locally finite if the language is relational.} countably infinite structure. If $Age(\M)$ has the canonical point-Ramsey property, then $\M$ has no nontrivial\footnote{Trivial equivalence relations are the one with just one equivalence class and the one with only equivalence class of size 1.} $\emptyset$-quantifier-free-type definable\footnote{Namely, for any $a,b\in M$, $a$ and $b$ are in the same equivalence class if and only if $(a,b)\vDash p(x,y)$ where $p(x,y)$ is a quantifier-free type without parameters.} equivalence relation.  
\end{proposition}

\begin{proof}
   Suppose $\M$ has a nontrivial $\emptyset$-quantifier-free-type definable equivalence relation, say defined by $p(x,y)$ where $p(x,y)$ is a quantifier-free type without parameters. Let $\{a_i\mid i<k\}$ be a maximal set of representatives where $k\in \mathbb{N}_{\geq 2}\cup\{\omega\}$ (as assumed nontrivial). Now suppose $Age(\M)$ has the canonical point-Ramsey property. As the equivalence relation is nontrivial, one can find $a\in M$ such that $a\not=a_0$ and $a, a_0$ are in the same equivalence class. By local finiteness, we have $B=\langle a,a_0,a_1\rangle$ is a finite substructure of $\M$ and thus $B\in Age(\M)$. Moreover, as $p(x,y)$ is quantifier-free type with no parameter, we have that $B\vDash p(a,a_0)$ and $B\not\vDash p(a_0,a_1)$. Now let $C\in Age(\M)$ be a witness of the canonical point-Ramsey property for $B$. Without loss of generality we can assume that $C\subseteq \M$. We then define the following coloring $\chi: C\rightarrow \mathbb{N}$ as follows: $\chi(c)= i$ if $c$ and $a_i$ belong to the same equivalence class. Now if $C$ contains a monochromatic copy of $B$, then all those elements should belong to the same equivalence class and thus any two of them should satisfy $p(x,y)$, which is impossible as $B$ contains two elements not satisfying $p(x,y)$; if $C$ contains a heterochromatic copy of $B$, then those elements belong to different equivalence class and thus any two distinct element does not satisfy $p(x,y)$, which is impossible as $B$ contains two distinct elements satisfying $p(x,y)$. Therefore, there is no monochromatic or heterochromatic copy of $B$ inside $C$, a contradiction. Therefore, $Age(\M)$ does not have the canonical point-Ramsey property. This finishes the proof.
\end{proof}

By Proposition \ref{definable equi}, we get the following corollary immediately:

\begin{corollary}
    Let $\M$ be a locally finite countably infinite structure. If $Age(\M)$ has the finite 2-sunflower property, then $\M$ has no nontrivial $\emptyset$-quantifier-free-type definable equivalence relation.  
\end{corollary}

Using the above corollary, we have that the finite 2-sunflower property fails for the first two examples in \cite[Exam 3.22]{2} and $nK_{\omega}$ (i.e., a disjoint union of $n$ copies of the complete graph $K_\omega$) where $n\geq 2$ for free. Besides, although it is unknown whether the finite-sunflower property is equivalent to the canonical point-Ramsey property or not, the result above certainly fails for classes of structures with the well-studied point-Ramsey property\footnote{See remarks after Proposition \ref{A[A]}.} in general (e.g., the class of finite equivalence relations).

Finally, we introduced an amalgamation property and show that any class of structures with this amalgamation property has the finite sunflower property. 

In \cite{2}, the authors defined what is called the \textit{very canonical point-Ramsey property} and used it to show the finite sunflower property for some Fra\"iss\'e classes.

\begin{definition}\cite[Def 3.1]{2}
    Let $\K$ be a class of finite structures. $\K$ has the \textit{very canonical point-Ramsey property} if for any $s\in \mathbb{N}^+$, any $B\in\K$, there is $C\in\K$ with a partition $V_0\sqcup ... \sqcup V_{|B|-1}$ of it such that for any $s$ colorings $(\chi_i: C\rightarrow \omega)_{i<s}$, at least one of the following holds: 
\begin{itemize}
    \item  there is $i<s$ and $B'\subseteq C$ with $B'\cong B$ such that $\chi_i$ is monochromatic on $B'$;
    \item  there is a transversal (i.e., the intersection with each $V_i$ is singleton) $B'\subseteq C$ with $B'\cong B$ such that for all $i<s$, the coloring $\chi_i$ is heterochromatic on $B'$.
\end{itemize}
     
\end{definition}

\begin{definition}
    A \textit{Berge cycle} of length $k$ in a hypergraph
is a sequence $V_0,e_0,v_1,e_1,...,v_{k-1},e_{k-1}$ such that each $v_i\in e_{i-1},e_i \text{ mod } k$ where $v_i$'s are pairwise distinct vertices and $e_i$'s are pairwise distinct hyperedges. The
\textit{girth} of a hypergraph is the smallest length of its Berge cycles (if there is no Berge cycle, then it is $\infty$).
\end{definition}

It was shown that the very canonical point-Ramsey property implies the finite sunflower property \cite[Thm. B]{2}. And they managed to prove the very canonical point-Ramsey property and thus the finite sunflower property for transitive\footnote{Namely, there is a unique quantifier-free type for any element of a structure in $\K$.} free amalgamation classes in relational languages and $Age(\U_V)$\footnote{Defined in Section 3.2.}s, using the following proposition. The proof is a probabilistic argument.

\begin{proposition}\cite[Prop 3.3]{2}
Let $n,s,g\in\mathbb{N}$ with $n,g\geq 2$ and $s\geq 1$. Then there is a finite $n$-uniform hypergraph $H$ with girth $\geq g$ and with a partition $V_0\sqcup ... \sqcup V_{n-1}$ of its vertex set such that for any $s$ vertex-colorings $(\chi_r: C\rightarrow \omega)_{r<s}$, at least one of the following holds: 
   \begin{itemize}
    \item  there is $r<s$ and $i<n$ such that $\chi_r$ is monochromatic on some edge contained in $V_i$;
    \item  there is a transversal edge $e$ such that for all $r<s$, the coloring $\chi_r$ is heterochromatic on $e$.
\end{itemize} 
\end{proposition}

It is conjectured that for a Fra\"iss\'e class, the very canonical property is equivalent to the finite sunflower property \cite[Conj 4.1]{2}. Motivated by the proofs in \cite{2}, we propose an amalgamation property as follows:

\begin{definition}
    Let $\K$ be a class of finite structures. We say that $\K$ has \textit{large-girth amalgamation} if for any $B\in\K$ with an enumeration without repetition $(b_i)_{i<|B|}$ where $|B|\geq 2$, there is $n_B\in\mathbb{N}^+$ such that for any $p(\Bar{x}_1)=...=p(\Bar{x}_n)=\mathrm{qftp}((b_i)_{i<|B|})$ where $\Bar{x}_i$'s are tuples of variables with $|\Bar{x}_i\cap \Bar{x}_j|<|B|$ for any $i\not=j$ and the smallest length of Berge cycles in $\Bar{x}_i$'s is $\geq n_B$\footnote{Namely, the girth number of the $|B|$-uniform hypergraph with each $\Bar{x}_i$ viewed as a hyperedge (i.e., the set $\{x_{i1},...,x_{i|B|} \}$ where $\Bar{x}_i=(x_{i1},...,x_{i|B|})$) is no smaller than $n_B$. The assumption $|\Bar{x}_i\cap \Bar{x}_j|<|B|$ for any $i\not=j$ simply means that $\Bar{x}_i$'s give different hyperedges.}, there is $C\in \K$ realizing $p(\Bar{x}_1)\cup...\cup p(\Bar{x}_n)$\footnote{For model theorists, they prefer to write $\mathrm{qftp}(B)$ instead of $\mathrm{qftp}((b_i)_{i<|B|})$ with one fixed enumeration of $B$ in mind. Following this notation, the conclusion can be restated as $p(\Bar{x}_1)\cup...\cup p(\Bar{x}_n)\subseteq \mathrm{qftp}(C)$.}. 
\end{definition}

\begin{proposition}
    Let $\K$ be a class of finite structures. If $\K$ has large-girth amalgamation, then $K$ is transitive.
\end{proposition}

\begin{proof}
    Suppose $\K$ has large-girth amalgamation while $K$ is not transitive, there is $B\in \K$ with $a_1,a_2\in B$ such that $\mathrm{qftp}(a_1)\not=\mathrm{qftp}(a_2)$. Now enumerate $B$ as $a_1, a_2, b_1,...,b_n$ without repetition. Now consider two quantifier-free types $p(x_1,x_2,y_1,...,y_n)=\mathrm{qftp}(a_1, a_2, b_1,...,b_n)$ and $p(x_2,z,w_1,...,w_n)=\mathrm{qftp}(a_1, a_2\\, b_1,...,b_n)$ where $x_1,x_2,y_1,...,y_n,z,w_1,...,w_n$ are pairwise distinct variables. Hence the girth number of the $|B|$-hypergraph with $\{x_1,x_2,y_2,...,y_n\}$ and $\{x_2,z,\\w_1,...,w_n\}$ as its edge is $\infty$ (as there is no Berge cycle). Hence by assumption, there is $C\in\K$ such that $C$ satisfies $p(x_1,x_2,y_1,...,y_n)\cup p(x_2,z,w_1,...,w_n)$. However, this is impossible as $p(x_1,x_2,y_1,...,y_n)\cup p(x_2,z,w_1,...,w_n)$ is inconsistent. Therefore, if $\K$ has large-girth amalgamation, then it is transitive.
\end{proof}

Now we can show that large-girth amalgamation actually gives us the very canonical point-Ramsey property and thus the finite sunflower property.

\begin{theorem}\label{large-girth}
    Let $\K$ be a class of finite structures. If $\K$ has large-girth amalgamation, then $\K$ has the very canonical point-Ramsey property and hence has the finite sunflower property.
\end{theorem}

\begin{proof}
    Suppose $\K$ has large-girth amalgamation. For any $B\in \K$, any $s\in\mathbb{N}^+$, let $n_B$ be the number given by Definition 4.8. Without loss of generality we can assume that $n_B\geq 2$ and $|B|\geq 2$. Now say $|B|=n$ with $b_1,...,b_n$ be an enumeration of $B$ without repetition, by Proposition 4.7, there is a finite $n$-uniform hypergraph $H$ with girth $\geq n_B$ and with a partition $V_0\sqcup ... \sqcup V_{n-1}$ of its vertex set such that for any $s$ vertex-colorings $(\chi_r: H\rightarrow \omega)_{r<s}$, at least one of the following holds: 
   \begin{itemize}
    \item  there is $r<s$ and $i<n$ such that $\chi_r$ is monochromatic on some edge contained in $V_i$;
    \item  there is a transversal edge $e$ such that for all $r<s$, the coloring $\chi_r$ is heterochromatic on $e$.
 \end{itemize}   

Without loss of generality we can assume that every element of $H$ is contained in at least one edge (as deleting those elements won't affect the above property that it has.). Now let $a_1,...,a_m$ enumerate all elements of $H$ without repetition. We introduce pairwise variables $x_1,...,x_m$ correspondingly. Now for each edge $\{a_{i_1},...,a_{i_n}\}$ in $H$ where $i_1<i_2<...<i_n$, we define $p(x_{i_1},...,x_{i_n}):=\mathrm{qftp}(b_1,...,b_n)$. Let $\Gamma=\{p(x_{i_1},...,x_{i_n})\mid \{a_{i_1},...,a_{i_n}\} \textit{ is an edge of }H \textit{ with }i_1<i_2<...<i_n\}$. As the girth number of $H$ is $\geq n_B$, by the definition of $n_B$, we know that there is $C\in \K$ such that $C$ realizes $\bigcup\Gamma$. As every element of $H$ is contained in at least one edge, we can assume without loss of generality that $dom(H)\subseteq C$  and $(a_{i_1},...,a_{i_n})\vDash p(x_{i_1},...,x_{i_n})$ in $C$ where $\{a_{i_1},...,a_{i_n}\}$ is an edge in $H$ with $i_1<i_2<...<i_n$. Now let $C=V'_0\sqcup ... \sqcup V'_{n-1}$ be a partition of $C$ such that $V_i\subseteq V'_i$ for any $i$. For any $s$ vertex-colorings $(\chi_r: C\rightarrow \omega) _{r<s}$ of $C$, we have induced $s$ vertex-colorings $(\chi_r|_{dom(H)}: H\rightarrow \omega)_{r<s}$ of $H$.

By our assumption about $H$, two cases can happen. In the first case, there is $r<s$ and $i<n$ such that $\chi_r|_{dom(H)}$ is monochromatic on some edge contained in $V_i$. Thus, there is $r<s$ and $i<n$ and edge $\{a_{i_1},...,a_{i_n}\}$ such that $\chi_r|_{dom(H)}$ is monochromatic on $\{a_{i_1},...,a_{i_n}\}$ with $i_1<i_2<...<i_n$. In particular, $\chi_r$ is monochromatic on $\{a_{i_1},...,a_{i_n}\}$ with $i_1<i_2<...<i_n$, and as $(a_{i_1},...,a_{i_n})\vDash p(x_{i_1},...,x_{i_n})$ in $C$, we also know that $\{a_{i_1},...,a_{i_n}\}\cong B$. Thus there is $r<s$ such that $C$ contains a monochromatic copy of $B$ with respect to the coloring $\chi_r$. In the second case, there is a transversal edge $\{a_{i_1},...,a_{i_n}\}$ with $i_1<i_2<...<i_n$ of $H$ such that for all $r<s$, $\chi_r|_{dom(H)}$ is heterochromatic on $\{a_{i_1},...,a_{i_n}\}$. As $V_i\subseteq V'_i$ for any $i$, $\{a_{i_1},...,a_{i_n}\}$ is also transversal with respect to the partition $C=V'_0\sqcup ... \sqcup V'_{n-1}$ and for all $r<s$, $\chi_r$ is heterochromatic on $\{a_{i_1},...,a_{i_n}\}$. As $(a_{i_1},...,a_{i_n})\vDash p(x_{i_1},...,x_{i_n})$ in $C$, we again have  that $\{a_{i_1},...,a_{i_n}\}\cong B$. Namely, $C$ has a transversal copy of $B$ which is heterochromatic with respect to all $\chi_r$'s. Therefore, $C$ witnesses the very canonical point-Ramsey property for $B$ with respect to to $s$. In conclusion, $\K$ has the very canonical point-Ramsey property and hence has the finite sunflower property.  
 
\end{proof}

With large-girth amalgamation in mind, one can check that what the authors in \cite{2} actually showed is  that every transitive free amalgamation class in a relational language has large-girth amalgamation, and every $Age(\U_V)$ has large-girth amalgamation. In particular, their proofs tell us that for the first one, $n_B$ is independent of $B$ ($n_B=4$) while for the second one, there is no such uniform $n_B$ in general \cite[Prop 3.9 and 3.17]{2}. Therefore, the above theorem uniformizes their results. 

However, it is easy to show that if there is a total order on structures of $\K$, then $\K$ does not have large-girth amalgamation (no matter how big $n_B$ is, we can create a large enough cycle to get inconsistency). Thus we shouldn't expect that large-girth amalgamation covers every instance of classes with the finite sunflower property.  Maybe what is more interesting is that by isolating large-girth amalgamation, we can study it independently using tools from model theory. As this will lead to things much more model-theoretic, we leave it for future research.

Finally, we end this paper with some concrete examples with the finite sunflower property and the canonical point-Ramsey property. We first give two applications of Theorem \ref{large-girth}. In the following, all graphs $G=(V(G), E^G)$ are undirected simple graphs where $V(G)$ is the vertex set while $E^G$ is the edge relation. As people often do in this area, we may also use $G$ to denote its vertex set when no confusion is caused. 

Let $P_n$ be the path on $n$ vertices and let $C_n$ be the cycle on $n$ vertices where $n\geq 3$. For a class of finite graphs $\mathcal{A}$, we use $H_\mathcal{A}$ to denote the class of all finite $\mathcal{A}$-free graphs.

\begin{proposition}\label{cycle}
Let $\A\subseteq \{P_n\mid n\geq 4\}\cup\{C_n\mid n\geq 5\}$. Then $H_\A$ has large-girth amalgamation and thus has the finite sunflower property and the canonical point-Ramsey property.
\end{proposition}

\begin{proof}
Let $B\in H_\A$ be arbitrary with an enumeration without repetition $(b_1,...,b_n)$ where $|B|=n\geq 2$, and let $n_B=3$. For any $p(\Bar{x}_1)=...=p(\Bar{x}_k)=\mathrm{qftp}((b_i)_{i<n})$ where $|\Bar{x}_i\cap \Bar{x}_j|<|B|$ for any $i\not=j$ and the smallest length of Berge cycles in $\Bar{x}_i$'s is $\geq n_B$, we construct a graph $G$ as follows: for any $1\leq i\leq k$, let $V(G_i)=\{x_{i1},...,x_{in}\}$ where $\Bar{x}_i=(x_{i1},...,x_{in})$ and $(x_{ij},x_{ij'})\in E^{G_i}$ holds if and only if $E(x_{ij},x_{ij'})\in p(\Bar{x}_i)$ for any $1\leq j,j'\leq n$. Let $V(G)=\bigcup_{1\leq i\leq k}V(G_i)$ and $E^G= \bigcup_{1\leq i\leq k} E^{G_i}\cup\{(x_{ij},x_{i'j'})\mid \text{there is no } l \text{ such that } x_{ij}, x_{i'j'} \text{ both  appear in } G_l\}$. Namely, $G=(V(G), E^G)$ is the union of $G_i$'s added with all edges connecting elements from different $G_i$'s. Clearly, $G_i\cong B$ for any $1\leq i\leq k$. As the smallest length of Berge cycles in $\Bar{x}_i$'s is $\geq 3$, we know that $|V(G_i)\cap V(G_j)|\leq 1$ for $i\not =j$. Hence $G|_{V(G_i)}=G_i$ for any $1\leq i\leq k$, and $G$ realizes $p(\Bar{x}_1)\cup...\cup p(\Bar{x}_k)$.

 Now suppose $G\not\in H_\A$, then there is $A\in \A$ such that $G$ contains a copy of $A$. As $\A \subseteq \{P_n\mid n\geq 4\}\cup\{C_n\mid n\geq 5\}$, $A=P_m$ for some $m\geq 4$ or $A=C_m$ for some $m\geq 5$. In the first case, there exist $a_1,...,a_m\in G$ such that $(a_i,a_j)\in E^G$ if and only if $|i-j|=1$ for any $1\leq i,j\leq m$ where $m\geq 4$. As $a_1$ is not adjacent to any one of $a_3,a_4,...,a_m$, we know that $a_1,a_3,a_4,...,a_m$ must be in the same $G_i$ by our construction of $G$. And since $a_2$ is not adjacent to any one of $a_4,...,a_m$, we have that $a_2,a_4,...,a_m$ are in the same $G_i$. Therefore, $a_1,a_2,...,a_m$ must belong the same $G_i$, which means that $G_i$ contains a copy of $P_m$. However, as $G_i\cong B$ while $B\in H_\A$, we get a contradiction. In the second case, there exist $a_1,...,a_m\in G$ such that $(a_i,a_j)\in E^G$ if and only if $|i-j|=1 \text{ mod } m$ where $m\geq 5$.  As $a_1$ is not adjacent to any one of $\{a_3,...a_{m-1}\}$, they belong to the same $G_i$ by our construction of $G$. Similarly, as $a_2$ is not adjacent to $a_{m-1}$, and $a_3$ is not adjacent to $a_m$ (as $m\geq 5$), we know that $a_2,a_{m-1}$ are in the same $G_i$ and $a_3, a_{m}$ are in the same $G_i$. Therefore, $a_1,...,a_m$ are all in the same $G_i$ which means that $G_i$ contains a copy of $C_m$. As $G_i\cong B$ while $B\in H_\A$, we have a contradiction. Since we get a contradiction in either case, $G\in H_\A$. As $G$ realizes $p(\Bar{x}_1)\cup...\cup p(\Bar{x}_k)$ and $B\in H_\A$ is arbitrary, this proves that $H_\A$ has large girth amalgamation. Hence by Theorem \ref{large-girth} and Proposition 4.3, we have that  $H_\A$ has the finite sunflower property and the canonical
point-Ramsey property.

\end{proof}

It is easy to see that the finite sunflower property is preserved by taking the complement\footnote {For a graph $A$, the complement of $A$ (denoted as $\overline{A}$) is the graph with the same vertex set while $(a,b)\in E^{\overline{A}}$ if and only if $(a,b)\not\in E^A$ for any $a\not=b \in A$. For any a class of graphs $\K$, the complement of $\K$ is the class $\{\overline{A}\mid A\in \K\}$.} of a class of graphs. We thus have the following as well:

\begin{corollary}
    Let $\A\subseteq \{\overline{P_n}\mid n\geq 4\}\cup\{\overline{C_n}\mid n\geq 5\}$. Then $H_\A$ has the finite sunflower property and the canonical point-Ramsey property.
\end{corollary}

In particular, as the class of all finite cographs is $H_{\{P_4\}}$, we get another easy corollary:

\begin{corollary}
    The class of all finite cographs has the finite sunflower property and thus the canonical point-Ramsey property.
\end{corollary}

Besides, we can actually generalize Proposition \ref{cycle} much further with the help of a new definition: for any graph $G=(V(G),E^G)$, we call it \textit{non-2-covered} if for any $x,y\in G$, $N(x)\cup N(y)\not= V(G)$ (i.e., $V(G)$ is not the union of the neighborhood of two elements). The proof of the following theorem is almost the same as that of Proposition \ref{cycle}: if the graph $G$ we construct for $B$ in the proof is not $\A$-free, then one hyperedge (isomorphic to $B$) is not $\A$-free by non-2-coveredness. 

\begin{theorem}\label{non-2}
    Let $\A$ be a class of finite non-2-covered graphs, then $H_\A$ has the finite  sunflower property and the canonical point-Ramsey property.
\end{theorem}

However, one may wonder what if $\A$ includes $C_3$ or $C_4$ (none of them is non-2-covered). It turns out that we can prove a similar result in that case if $\A$ is finite and contains only cycles:

\begin{proposition}
    Let $\A\subseteq \{C_n\mid n\geq 3\}$ be finite. Then $H_\A$ has large-girth amalgamation and thus has the finite sunflower property and the canonical point-Ramsey property.
\end{proposition}

\begin{proof}
    Let $m$ be the largest number such that $C_m\in \A$. For any $B\in H_\A$, let $n_B=m+3$ and $b_1,...,b_n$ be an enumeration of $B$ without repetition. For any $p(\Bar{x}_1)=...=p(\Bar{x}_k)=\mathrm{qftp}((b_i)_{i<n})$ where $|\Bar{x}_i\cap \Bar{x}_j|<n$ for any $i\not=j$ and the smallest length of Berge cycles in $\Bar{x}_i$'s is $\geq n_B$, we construct a graph $G$ as follows: for any $1\leq i\leq k$, let $V(G_i)=\{x_{i1},...,x_{in}\}$ where $\Bar{x}_i=(x_{i1},...,x_{in})$ and $(x_{ij},x_{ij'})\in E^{G_i}$ holds if and only if $E(x_{ij},x_{ij'})\in p(\Bar{x}_i)$ for any $1\leq j,j'\leq n$. Let $V(G)=\bigcup_{1\leq i\leq k}V(G_i)$ and $E^G= \bigcup_{1\leq i\leq k} E^{G_i}$. Namely, $G=(V(G), E^G)$ is the union of $G_i$'s. Clearly, $G_i\cong B$ for any $1\leq i\leq k$.  As the smallest length of Berge cycles in $\Bar{x}_i$'s is $\geq n_B\geq 3$, we know that $|V(G_i)\cap V(G_j)|\leq 1$ for $i\not =j$. Hence $G|_{V(G_i)}=G_i$ for any $1\leq i\leq k$. Thus $G$ realizes $p(\Bar{x}_1)\cup...\cup p(\Bar{x}_k)$. 

    Now suppose $G\not\in H_\A$. Then there is $C_l\in \A$ such that $G$ contains a copy of $C_l$ (note $l\leq m < n_B$). Namely there exist $a_1,...,a_l\in G$ such that $(a_i,a_j)\in E^G$ if and only if $|i-j|=1\text{ mod } l$. For each $(a_j,a_{j'})$ where $|j-j'|=1\text{ mod } l$, we have $(a_j,a_{j'})\in E^G$ and thus $\{a_j,a_{j'}\}\subseteq G_i$ for some (unique) $i$ by our definition of $G$. Hence we have $a_1G_{i1}a_2...a_{il-1}G_{i{l-1}} a_l G_{il}a_1$ where $\{a_1,a_2\}\subseteq G_{i1}$,...,$\{a_{l-1}, a_l\}\subseteq G_{i{l-1}},\{a_l, a_1\}\subseteq G_{il}$ and $(a_1,a_2)\in E^{G_{i1}}$,..., $(a_{l-1},a_l)\in E^{G_{il-1}}, (a_l,a_1)\in E^{G_il}$. We know that  $i_1,...,i_l$ can't all be the same as $G_i\cong B$ for any $1\leq i\leq k$ and $B$ is $C_l$-free. Hence $a_1G_{i1}a_2...G_{i{l-1}} a_l G_{il}a_1$ gives us a Berge cycle of length $\leq l<n_B$, contradicting our assumption about $\Bar{x}_i's$. Therefore, $G\in H_\A$. As $G$ realizes $p(\Bar{x}_1)\cup...\cup p(\Bar{x}_k)$ and $B\in H_\A$ is arbitrary, this proves that $H_\A$ has large girth amalgamation. Hence by Theorem \ref{large-girth} and Proposition 4.3, we have that $H_\A$ has the finite sunflower property and the canonical
point-Ramsey property.
    \end{proof}

\begin{corollary}
    Let $\A\subseteq\{\overline{C_n}\mid n\geq 3\}$ be finite. Then $H_\A$ has the finite sunflower property and the canonical point-Ramsey property.
\end{corollary}

In fact, for the canonical point-Ramsey property, we can get a more general result by adopting the proof strategy in \cite{g}, which strengthens one of their main results \cite[Prop 4.6]{g}. We start with the following definition, which is the finite version of indivisibility\footnote{See Definition \ref{indiv}.}.

\begin{definition}
    A class of finite structures $\K$ is indivisible if and only if for all $A\in \K$ and positive integers $k$, there is $B\in \K$ such that for every $k$-coloring of $B$, $B$ has a monochromatic copy of $A$. 
\end{definition}

\begin{definition} Let $L$ be a relational language where each relational symbol is at least binary.
    For any finite irreflexive\footnote{Namely, for any $R\in L$ of arity $n\geq 2$, $(a_1,...,a_n)\in R^C$ implies that $a_i$'s are pairwise distinct.} $L$-structures $A$ and $B$, the \textit{lexicographic product} of $A$ and $B$, denoted as $A[B]$, is the $L$-structure whose underlying set is $A\times B$ with $R^{A[B]}$ where $R\in L$ of arity $n$ defined as follows: for any $(a_1,b_1),..., (a_n,b_n)\in A\times B$, $((a_1,b_1),...,(a_n,b_n))\in R^{A[B]}$ if and only if either $(a_1,...,a_n)\in R^A$ or $a_1=...=a_n$ and $(b_1,...,b_n)\in R^B$.
\end{definition}

With the help of the above construction, we are able to show the following proposition for the canonical point-Ramsey property.

\begin{proposition}\label{A[A]}
    Let $L$ be a relational language where each relational symbol is at least binary and let $\K$ be a class of irreflexive finite $L$-structures. If for any $A\in \K$, $A[A]\in \K$, then $\K$ has the canonical point-Ramsey property.
\end{proposition}

\begin{proof}
    It is easy to check that $\K$ is indivisible \cite[Thm 4.2]{g}. Let $A\in \K$ be arbitrary, by indivisibility of $\K$, we have $B\in \K$ such that for any $(|A|-1)$-coloring of $B$, B has a monochromatic copy of $A$ inside $B$ (without loss of generality assume $|A|\geq 2$). Say $A=\{a_1,...,a_n\}$ where $n\geq 2$. Now for any coloring $\chi: A[B]\rightarrow \mathbb{N}$, suppose $B$ has no monochromatic copy of $A$. We recursively construct a heterochromatic copy of $A$. Pick $(a_1,b_1)\in A[B]$. Suppose $(a_1,b_1),...,(a_i,b_i)$ have already be chosen such that their colors are pairwise distinct where $1\leq i<n$. If $|\chi(\{a_{i+1}\}\times B)|\leq i\leq |A|-1$, then by our assumption about $B$, we get a monochromatic copy of $A$ inside $\{a_{i+1}\}\times B$ as $\{a_{i+1}\}\times B\cong B$, which contradicts our assumption. Therefore, $|\chi(\{a_{i+1}\}\times B)|>i$. We pick $(a_{i+1},b_{i+1})\in \{a_{i+1}\}\times B$ such that $\chi((a_{i+1},b_{i+1}))$ is different any one of $\chi((a_1,b_1)),..., \chi((a_i,b_i))$. Continuing this process, in the end we get $(a_1,.b_1),...,(a_n,b_n)\in A[B]$ such that their colors are pairwise distinct. By the definition of $A[B]$, one can easily see that $\{(a_1,b_1),...,(a_n,b_n)\}\cong A$. Therefore, we have a heterochromatic copy of $A$ inside $A[B]$. Thus for any coloring $\chi: A[B]\rightarrow \mathbb{N}$, there is a monochromatic or heterochromatic copy of $A$ inside $A[B]$. Clearly $A$ is embeddable into $B$, and thus $A[B]$ is embeddable into $B[B]$. Therefore, for any $\chi: B[B]\rightarrow \mathbb{N}$, there is a monochromatic or heterochromatic copy of $A$ inside $B[B]$ as well. As $B\in\K$, by the assumption we have that $B[B]\in \K$. As $A\in \K$ is arbitrary, this proves that $\K$ has the canonical point-Ramsey property. 
\end{proof}

As indicated above \footnote{See remarks after Corollary 4.5.}, indivisibility does not imply the the canonical point-Ramsey property in general\footnote{For graphs, take the class of all finite graphs which are disjoint unions of complete graphs (i.e., $Age(\omega K_\omega)$) as an example.}. Thus the above result does not follow from \cite[Thm 4.2]{g}. And since the canonical point Ramsey property clearly implies indivisibility (as long as every structure of size 2 in the class can be embedded into one bigger structure in that class), Proposition \ref{A[A]} is stronger than \cite[Thm 4.2]{g}. 

By the above result, we also get the following strengthening of \cite[Prop 4.6]{g}.

\begin{proposition}
    Let $\A\subseteq \{P_n\mid n\geq 4\}\cup\{C_n\mid n\geq 5\}\cup \{\overline{P_n}\mid n\geq 4\}\cup\{\overline{C_n}\mid n\geq 5\}$. Then $H_\A$ has  the canonical point Ramsey property.
\end{proposition}

\begin{proof}
    By \cite[Lem 2.7, 4.4 and 4.5]{g}, we know that for any $A,B\in H_\A$, $A[B]\in H_\A$. Hence, by Proposition \ref{A[A]}, we get the result.
\end{proof}

In particular, as the class of all finite perfect graphs is $H_\A$ where $\A =\{C_n\mid n\geq 5\text{ is odd}\}\cup \{\overline{C_n}\mid n\geq 5\text{ is odd}\}$ by \cite[Thm 1.2]{p}, we have a strengthening of \cite[Cor 4.11]{g}:

\begin{proposition}
     The class of all finite perfect graphs has  the canonical point Ramsey property.
\end{proposition}

\noindent\textbf{Acknowledgments}.      The author is very grateful to his supervisor Nick Ramsey for helpful comments.

\end{document}